\documentclass{article}               
\usepackage{float} 
\usepackage{graphicx}
\usepackage{enumerate}
\usepackage{amssymb, amsmath, amsthm}
\usepackage{dsfont}
\usepackage{textcomp}
\usepackage{listings}
\usepackage{color} 
\usepackage{hyperref}

\newtheorem{proposition}{Proposition}[section] %denna med i latex
\newtheorem{defn}[proposition]{Definition}
\newtheorem{assumption}[proposition]{Assumption}
\newtheorem{thm}[proposition]{Theorem}
\newtheorem{lemma}[proposition]{Lemma}
\newtheorem{corollary}[proposition]{Corollary}
\newtheorem{remark}[proposition]{Remark}

\usepackage[T1]{fontenc}
\usepackage{lmodern}

\newcommand{\N}{\ensuremath{{\mathbb N}}}

\newcommand{\R}{\ensuremath{{\mathbb R}}}

\newcommand{\hta}{\ensuremath{{\hat \tau}}}

\newcommand{\E}{\ensuremath{{\mathbb E}}}
\newcommand{\h}{\ensuremath{{ h}}}
\newcommand{\Pro}{\ensuremath{{\,\mathbb P}}}

\usepackage{mathtools}
\DeclarePairedDelimiter{\ceil}{\lceil}{\rceil}
\def\SpecialChap#1{{\let\cleardoublepage\relax\chapter{#1}}}

\begin{document}
	\title{General Optimal Stopping with Linear Costs}
\author{S\"oren Christensen\thanks{Christian-Albrechts-Universit\"at zu Kiel, Department of Mathematics, Ludewig-Meyn-Str. 4,
24118 Kiel, Germany} \and Tobias Sohr\footnotemark[1]}
\maketitle

{\small \noindent\textbf{Abstract:} }This article treats both discrete time and continuous time stopping problems for general Markov processes on the real line with general linear costs. Using an auxiliary function of maximum representation type, conditions are given to guarantee the optimal stopping time to be of threshold type. The optimal threshold is then characterized as the root of that function. For random walks our results condense in the fact that all combinations of concave increasing pay-off functions and convex cost functions lead to a one-sided solution. For L\'evy processes an explicit way to obtain the auxiliary function and the threshold  is given by use of the ladder height processes. Lastly, the connection from discrete and continuous problem and possible approximation of the latter one via the former one is discussed.

{\small \textbf{Keywords:}  Optimal Stopping; Discrete time; Continuous Time; Maximum representation; Threshold Times; Monotone Stopping Rules.}

{\small \textbf{Subject Classifications:}60G40; 62L15.}

\section{INTRODUCTION}

Optimal stopping problems with running costs of observation arise frequently in sequential decision making, see \cite{MR2083932} for an overview of many examples ranging from sequential statistics to finance. However, general solution techniques are only known for underlying diffusion processes (\cite{MR2083932,cisse2012optimal}) or certain subclasses of problems (\cite{MR1286327,MR1626790}). The main aim of this article is to close this gap. 

The main tool for this is the well-known notion of monotone stopping problems that was developed back in the early days of sequential decision making, see e.g. \cite{ChowRobins61}, and used extensively throughout the following decades. More recently, the power of this easy line of argument for more advanced problems was rediscovered, see \cite{MR3620898,CI19}, or \cite{ChristensenSohr} in the context of impulse control problems.
One idea in these work was to embed monotone problems in a priori non-monotone stopping settings. We adapt this technique to tackle (undiscounted) problems with generalized linear costs and characterize under which conditions the problems have a threshold time as an optimizer. 
The second part of this work tackles the continuous time analogous to the discrete time problem. The right tool to carry over the idea to utilize monotonicity turns out to be a maximum representation of the pay-off function. The root of a function occurring in said maximum representation yields the optimal threshold. In the case of L\'evy processes, we show how to find the maximum representation and hence optimal threshold (semi-)explicitly by use of the ladder height process. Because the generalized linear cost term may be tackled using a time change, these problems for L\'evy processes provide a versatile tool for applications. 

\section{DISCRETE STOPPING PROBLEM}\label{dicretesection}

Let $\left(Y_n\right)_{n\in \N}$ be a discrete time Markov chain on the real line. 
We aim to study the stopping problem for \begin{align}\gamma \left(Y_k\right)- \sum_{i=1}^{k}\  h\left(Y_i\right),\ \ k \in \N_0,\label{auxilliarystopping} \end{align} where $\gamma$ is a non-decreasing function and $h$ a non-negative non-decreasing function. \\ Namely we want to find the value function 
$$V\left(y\right):=\sup_{\tau \in \mathcal T}\E_y\left(\gamma\left(Y_\tau\right)- \sum_{i=1}^{\tau}\  \h\left(Y_i\right)\right)$$
where $\mathcal T$ is the set of all stopping times and we want to characterize when the stopping region 
$$S^*:= \{y\in \R| V\left(y\right)=\gamma\left(y\right)\} $$
is an interval of the form $\left[x,\infty\right)$ or $\left(x,\infty\right)$ for some $x \in \R$ and hence the stopping time
$$\tau^*:=\inf\{n|Y_n\in S^*\}$$ is a threshold time. 
To ensure that $\tau^*$ is an optimizer for $V$ and to exclude cases where waiting infinitely long would be optimal, we make one assumption that is pretty standard in optimal stopping. 
\begin{assumption}\label{opt}
	\begin{align*}
	\E_y\left(\sup_n\gamma\left(Y_n\right)- \sum_{i=1}^{n}\  \h\left(Y_i\right)\right) < \infty \ \forall y\in \R \\ \gamma\left(Y_n\right)-\sum_{i=1}^{n}\  \h\left(Y_i\right) \rightarrow -\infty \ \ a.s. 
	\end{align*}
\end{assumption}

\begin{lemma}
	The stopping time $\tau^*$ is the a.s. smallest optimal stopping time. 
\end{lemma}
\begin{proof}
	This is a standard result in optimal stopping. In \cite{shiryayev}, Chapter 2, Theorem 8 a similar result without running costs can be found. 
\end{proof}

A main ingredient in our line of argument will be the function $f$ defined by
\begin{align*}
f\left(y\right)=\frac{\phi\left(y\right)-\gamma\left(y\right)}{\E_y\left({\tau^+}\right)}
, \label{bedingunggleichung} \tag{$*$}
\end{align*}
for all $y \in \R$, where $\tau^+:=\inf\{t\geq 0|Y_t >Y_0 \}$ for all $y \in \R$ and $$
\phi\left(y\right):= \E_y\left(\gamma\left(Y_{\tau^+}\right)-\sum_{i=1}^{\tau^+}\  \h\left(Y_i\right)\right).
$$ for all $y \in \R$. In the following we will always assume that $f$ is a well defined real valued function. All the degenerate cases, where one or more of the occurring terms are infinite can be treated by simple arguments or dismissed as unrealistic. 
Heuristically $f$ being positive means that the gain one would get by waiting for the process to rise above the present level exceeds the possible pay-off at the present level. 
\begin{assumption}\label{bedingung}
	\begin{itemize}
		
		\item $f$ is a well defined function $\R \rightarrow \R$,
		\item there is exactly one $ \overline{x}\in \R$ such that  $f\left(x\right) > 0 $ for all  $x \in \left(-\infty, \overline x\right)$, and $f\left(x\right)< 0$ for all $ x\in \left(\overline x, \infty\right)$,
		\item  on $\left[\overline{x},\infty\right)$, the function $f$ is non-increasing.
	\end{itemize}
\end{assumption} 
\emph{From now on we will always assume \ref{opt} and \ref{bedingung} to be true.}\\ 
Under this assumption we are able to show that the first entrance time into $\left(\overline x,\infty\right)$ or $\left[\overline x, \infty\right)$ is the optimizer for  $V\left(y\right)$ for all $y \in  \R$. Which type of interval is the right choice depends on the value $f\left(\overline x\right)$. In the following, we assume $f\left(\overline x\right) > 0$ and show that the first entry time in $\left(\overline x,\infty\right)$ is optimal. With the same line of argument, one can show that the first entry time into $\left[\overline x, \infty\right)$ is optimal  if $f\left(\overline x\right) \leq 0$. \\

\begin{remark}
	While the numerator has an obvious interpretation, the necessity of dividing by the expected waiting time for the process to exceed the present level is not entirely obvious (although, especially if we compare the discrete time case with the continuous time case, it forges links to the generator of the ladder height process occurring therein). And, indeed, if the function $f$ as in Assumption \ref{bedingung} fails to be monotone after the root $\bar x$, one can instead work with a function $$\tilde f(y):=\frac{\phi\left(y\right)-\gamma\left(y\right)}{\E_y\left(\sum_{i=1}^{\tau^+}g\left(Y_i\right)\right)}=f(y)\frac{\E_y(\tau^+)}{\E_y\left(\sum_{i=1}^{\tau^+}g\left(Y_i\right)\right)}$$ for some positive function $g$. All the later proofs work with such a $\tilde f$ that fulfils Assumption \ref{bedingung} as well, nevertheless, for the sake of brevity and clarity, we will just use the 'standard' $f$. 
\end{remark}

%We notice that if $\E_y g\left(Y_n\right)- \rho n\right) \rightarrow \infty$ we can get an arbitrary large reward by waiting long enough, hence we assume $\sup_n \E_y g\left(Y_n\right)- \rho n\right) < \infty$. \\
%Take an $\Delta$ not in the state space of $Y$, set $Y_\infty:= \Delta$, $g\left(\Delta\right)=-\infty$. %and $\Delta + e=-\infty$ for each $e \in E$. 
For all $x \in \R$ we write $\tau_x:=\inf\{n \in \N|Y_n >x \}.$
The first step towards the proof  of the optimality of $\tau_{\overline{x}}$ is to show that $\tau_{\overline{x}}$ is optimal in the class of threshold times.  
\begin{lemma}\label{discreteintegraldst} Let $\sigma_0:=0$ and for all $n>0$ let $\sigma_n:=\inf\{k \in \N|X_k>X_{\sigma_n-1} \}$ be the $n$-th ladder time. Then  for all $x,y\in \R$ with $x \leq y$ we have $$ \E_x\left(\gamma(Y_{\tau_y})-\sum_{i=1}^{\tau_y}\h(Y_i)\right)=\gamma(x)+\E_{x}\left( \sum_{\sigma_n \leq \tau_y}f(Y_{\sigma_n})\E_{Y_{\sigma_n}}\left(\tau^+\right) \right)$$
\end{lemma}
\begin{proof}
	Expand $\E_x\left(\gamma(Y_{\tau_y})\right)-\gamma(x)$ in a telescoping series. 
\end{proof}
\begin{corollary}\label{thresholdcorrolary}
	The stopping time $\tau_{\overline x}$ is the optimizer for the original stopping problem amongst all threshold times. 
\end{corollary}
\begin{proof}
	Lemma \ref{discreteintegraldst} yields that in order to maximize $\E_x\left(\gamma(Y_{\tau_y})-\sum_{i=1}^{\tau_y}\h(Y_i)\right)$ in $y$, one has to sum as many positive summands of the form $f(Y_{\sigma_n})\E_{Y_{\sigma_n}}\left(\tau^+\right) $ on the right hand side as possible. Assumption \ref{bedingung} ensures that $\tau_{ \bar x}$ indeed yields the maximum. 
\end{proof}
\begin{thm}\label{discretetheorem}
	Let $y \in \R$. Then $$\tau_{\overline{x}}=\inf\{n \geq 0|Y_n > \overline{x} \}$$ is optimal for $V\left(y\right)$ for all $y \in \R$. 
	
\end{thm}
\begin{proof}
	As  Corollary \ref{thresholdcorrolary} indicates, the value function when only threshold times are admissible is given by $$
	\tilde V\left(y\right)=\begin{cases}
	\gamma\left(y\right) \ ; \ y > \overline{x} \\ \E_y\left(\gamma\left(Y_{\tau_{\bar x}}\right)-\sum_{i=1}^{\tau_{\bar x}}\  \h\left(Y_i\right)\right) \ ; \ y \leq \overline{x} 
	\end{cases}. $$
	Let $y \in \R$. We have $$\tilde{V}\left(y\right)\geq \gamma \left(y\right).$$ Hence, it remains to show $\h $-exessivity, meaning $$
	\tilde{V}\left(y\right)\geq \E_y\left(\tilde V\left(Y_1\right)-\h\left(Y_1\right)\right).
	$$   Define $$\tau':= \inf\{n>0|Y_n > \overline{x} \}.$$ By the strong Markov property we have \[
	\E_y\left(\tilde V\left(Y_1\right)-\h\left(Y_1\right)\right)=E_y\left(\gamma\left(Y_{\tau'}\right)-\sum_{i=1}^{\tau'}\  \h\left(Y_i\right)\right).
	\label{markov2} \tag{$**$} 
	\] 
	First look at the case $y \leq \overline{x}$: Here we have \\
	$$
	\tilde{V}\left(y\right)=\E_y\left(g\left(Y_{\tau'}\right)-\sum_{i=1}^{\tau'}\  \h\left(Y_i\right)\right) .
	$$  Further (\ref{markov2}) yields $$
	\E_y\left(\tilde V\left(Y_1\right)-\h\left(X_1\right)\right)=\E_y\left(\gamma\left(Y_{\tau'}\right)-\sum_{i=1}^{\tau'}\  \h\left(Y_i\right)\right)=\tilde V\left(y\right).
	$$ 
	Now assume $y > \overline{x}$.
	We define $\kappa_0:= 
	0$, $\kappa_1:=\inf\{n>0|Y_n>\overline{x}\}$ and for each $n >1$ we set $\kappa_n:= \inf\{n>\kappa_{n-1}|Y_n>Y_{\kappa_{n-1}}\}$. Again we use the notation $\tau^+:=\inf\{t\geq 0|Y_t>Y_0\}$. 
	Since this expression will occur later, we notice that \begin{align*}
	\sum_{n=0}^\infty  &\E_y\left(\E_{Y_{\kappa_n}}\left(\sum_{i=1}^{\tau^+}\  \h\left(Y_i\right)\right)1_{\{\kappa_n < \tau_y \}}\right)\\ &=\sum_{n=0}^\infty \E_y\left(\E_{y}\left(\sum_{i=\kappa_n+1}^{\kappa_{n+1}}\  \h\left(Y_i\right)|\mathcal A_{\kappa_n}\right)1_{\{\kappa_n < \tau_y \}}\right)
	\\&=\sum_{n=0}^\infty \E_y\left(\E_{y}\left(\left(\sum_{i=\kappa_n+1}^{\kappa_{n+1}}\  \h\left(Y_i\right)\right)1_{\{\kappa_n < \tau_y \}}|\mathcal A_{\kappa_n}\right)\right)
	\\&=\sum_{n=0}^\infty \E_y\left(\left(\sum_{i=\kappa_n+1}^{\kappa_{n+1}}\  \h\left(Y_i\right)\right)1_{\{\kappa_n < \tau_y \}}\right) 
	\\& = \E_y\left(\sum_{i=1}^{\tau_y}\  \h\left(Y_i\right)\right). \label{nebenrechnung}\tag{$*$}
	\end{align*}
	We obtain
	\begin{align*}
	\E_y&\left(\tilde V\left(Y_1\right)-\h\left(Y_1\right)\right)-\phi\left(y\right)=\E_y\left(\gamma\left(Y_{\tau'}\right)-\sum_{i=1}^{\tau'}\  \h\left(Y_i\right)\right)-\phi\left(y\right)\\&=\E_y\left(\gamma\left(Y_{\tau'}\right)-\gamma\left(Y_{\tau_y}\right)+\sum_{i=\tau'+1}^{\tau_y}\  \h\left(Y_i\right)\right)\\&=\sum_{n=0}^{\infty}\E_y\left(\left(\gamma\left(Y_{\kappa_n}\right)-\gamma\left(Y_{\kappa_{n+1}}\right)+\sum_{i=\kappa_n+1}^{\kappa_{n+1}}\  \h\left(Y_i\right)\right){1}_{\kappa_n < \tau_y}\right)\\&=\sum_{n=0}^{\infty}\E_y\left(\left(\gamma\left(Y_{\kappa_n}\right)-\left(\gamma\left(Y_{\kappa_{n+1}}\right)-\sum_{i=\kappa_n+1}^{\kappa_{n+1}}\  \h\left(Y_i\right)\right)\right) {1}_{\kappa_n < \tau_y}\right)
	\\&=\sum_{n=0}^{\infty}\E_y\left(\left(\gamma\left(Y_{\kappa_n}\right)-\E_y\left(\left(\gamma\left(Y_{\kappa_{n+1}}\right)-\sum_{i=\kappa_n+1}^{\kappa_{n+1}}\  \h\left(Y_i\right)\right)|\mathcal A_{\kappa_n}\right)\right) {1}_{\kappa_n < \tau_y}\right)
	\\&\stackrel{st. mkv}{=}\sum_{n=0}^{\infty}\E_y\left(\left(\gamma\left(Y_{\kappa_n}\right)-\E_{Y_{\kappa_n}}\left(\gamma\left(Y_{\tau^+}\right)-\sum_{i=1}^{\tau^+}\  \h\left(Y_i\right)\right) \right) {1}_{\kappa_n < \tau_y}\right)
	\\&=\sum_{n=0}^{\infty}\E_y\left(\left(\gamma\left(Y_{\kappa_n}\right)-\phi\left(Y_{\kappa_n}\right)\right) {1}_{\kappa_n<\tau_y}\right)
	\\& = \sum_{n=0}^\infty\E_y \left(\left(-f\left(Y_{\kappa_n}\right){\E_{Y_{\kappa_n}}\left(\sum_{i=1}^{\tau^+}\h\left(Y_i\right)\right)}\right)
	{1}_{\{\kappa_n < \tau_y\}}\right)
	\\& \stackrel{f \searrow}{\leq }
	\sum_{n=0}^\infty\E_y \left(\left(-f\left(y\right){\E_{Y_{\kappa_n}}\left(\sum_{i=1}^{\tau^+}\h\left(Y_i\right)\right)}\right)
	{1}_{\{\kappa_n < \tau_y\}}\right) 
	\\& =\sum_{n=0}^\infty\E_y\left( \left(\left(\gamma\left(y\right)-\phi\left(y\right)\right)\frac{\E_{Y_{\kappa_n}}\left(\sum_{i=1}^{\tau^+}\h\left(Y_i\right)\right)}{\E_y\left(\sum_{i=1}^{\tau_y}\h\left(Y_i\right)\right)}
	\right) {1}_{\{\kappa_n < \tau_y\}}\right)
	\\&\stackrel{\left(\bigstar\right)}{=}\gamma\left(y\right)-\phi\left(y\right)
	\\&=\tilde{V}\left(y\right)-\phi\left(y\right)
	\end{align*}
\end{proof}
\begin{remark}
	Since we are aiming for the one sided case, it of course seems natural to use the ascending ladder times to construct $f$. We consider it worth mentioning that when one instead uses 'skew' ladder times of the form $\tilde\tau^+:=\inf\{n\in\N|l(Y_n)>l(y) \}$ for some suitable (possibly non-monotonic) function $l$, the same proofs as beyond  still work and can lead to a characterization of the stopping sets by roots of an analogue function $f$ even in more complicated cases. However, the applicability of these more general result in concrete examples requires to 'guess' the right function $l$, which we only managed to do in trivial cases or situations which can be reduced to the one sided case anyway. So we decided to stick to the one sided situation in the proofs. 
\end{remark}
\subsection{Special Case: Random Walk}\label{random walk}
A nice application of our theory is the stopping problem for a random walk with linear costs.  
Assume $\left(X_i\right)_{i \in \N}$ is a sequence of independent, identically distributed random variables with $\Pro\left(X_1> 0\right)>0$, for all $n \in  \N_0$ let $$S_n:=\sum\limits_{i=1}^n X_i, $$ and $$Y_n^y:=y+S_n.$$ Denote with $\Pro_y$ the distribution of $Y^y$, set $\Pro:=\Pro_0$ and $\E:= \E_0$. %Further assume $\h$ is constant, hence for all $x\in \R$ we have $$\h(x)=c$$ for some $c>0$. \\
Then using the notation of the previous sections, we have \begin{align*}
\phi\left(y\right)&=\E_y\left(\gamma\left(Y_{\tau_y}\right)-\sum_{i=1}^{\tau_y}\  \h\left(Y_i^y\right)\right)
\\&=\E\left(\gamma\left(y+S_{\tau^+}\right)-\sum_{i=1}^{\tau^+}\  \h\left(y+S_i\right)\right).
%\\&=\E\left(\gamma\left(y+S_{\tau^+}\right)-c{\tau^+} \right)  
\end{align*}
%Now $y\mapsto \E_y\left(\tau_y\right)=\E\left(\tau^+\right)$ is constant, hence $f$ is increasing on $\left[\alpha^*,\infty\right)$, if and only if $\gamma-\phi$ is increasing on $\left[\alpha^*,\infty\right)$. 
Further $$\gamma\left(y\right)-\phi\left(y\right)=\E \left( \gamma\left(y\right)-\gamma\left(y+S_{\tau_+}\right)\right)-\E\left(\sum_{i=1}^{\tau_+}\  \h\left(y+Y_i\right)\right)$$ for all $y \in \R$.
Now $ y\mapsto \gamma\left(y\right)-\gamma\left(y+S_{\tau_+}\right)$ is increasing if $$y\mapsto \gamma\left(y\right) -\gamma\left(y+s\right) $$ is increasing for all $s>0$ or, equivalently, if $\gamma$ is convex. 
With the same argument we get that 
$$
y\mapsto -\E\left(\sum_{i=1}^{\tau_+}\  \h\left(y+Y_i\right)\right)
$$
is decreasing if $\h$ is non-decreasing. Hence, in the random walk case our findings read as follows:

\begin{thm}
	Assume that $Y$ is an random walk,
	define $$\bar x:=\inf\left\{y\in \R\ \Big| \E\left( \gamma\left(z\right)-\gamma\left(z+S_{\tau_+}\right)\right)\leq -\E\left(\sum_{i=1}^{\tau_+}\  \h\left(y+Y_i\right)\right) \right\}.$$ Further assume \ref{opt},  that $\gamma$ is concave on $\left[\bar x,\infty\right)$, and that $\h$ in non-decreasing and non-negative.   
	\begin{itemize}
		\item If $f(\bar x)\leq 0$, then $$\tau^*=\inf\{n\geq 0|Y_n\geq \bar x\}$$ is an optimal stopping time.
		\item If $f(\bar x) >0$, then $$\tau^*=\inf\{n\geq 0|Y_n> \bar x\}$$ is an optimal stopping time.  
	\end{itemize}
\end{thm}
\begin{corollary}
	Assume that $Y$ is a random walk. Further assume \ref{opt}, that $\gamma$ is concave on $\left[\bar x,\infty\right)$, and that $\h$ is constant, i.e. $\h(x)=c$ for some $c>0$. Define $$\bar x:=\inf\{y\in \R|\E \left( \gamma\left(z\right)-\gamma\left(z+S_{\tau_+}\right)\right)\leq -c \E\left({\tau_+}\right) \}.$$ Then we get:
	\begin{itemize}
		\item If $f(\bar x)\leq 0$, then $$\tau^*=\inf\{n\geq 0|Y_n\geq \bar x\}$$ is an optimal stopping time.
		\item If $f(\bar x) >0$, then $$\tau^*=\inf\{n\geq 0|Y_n> \bar x\}$$ is an optimal stopping time.  
	\end{itemize}
\end{corollary}
 Here we want to point out the connection to \cite{MR1286327} %and \cite{MR1626790}
who study traditional parking problems for random walks without running costs. If $h$ is constant, Wald's identity can be used to transform the problem to a stopping problem with no running costs. So these results can be viewed as a generalization of the results in the mentioned article. Also the line of argument here is inspired by the considerations there, but we avoid the use of the Wiener-Hopf factorization.

\section{CONTINUOUS PROBLEM}

This section translates the ideas from the discrete time case to the continuous time case. Since there are many technical difficulties when it comes to the definition of the ladder times in continuous time, it is by no means straightforward to define an analogous function $f$. To elaborate the similarities of both cases we hence shortly review the key properties of $f$: Lemma \ref{discreteintegraldst}  yields that for $x,y\in \R$ we have \begin{align*}
\E_x\left(\gamma(Y_{\tau_y})-\sum_{i=1}^{\tau_y}\h(Y_i)\right)=\gamma(x)+\E_{x}\left( \sum_{\sigma_n \leq \tau_y}f(Y_{\sigma_n})\E_{Y_{\sigma_n}}\left(\tau^+\right) \right).
\end{align*} The right hand side herein can be viewed as the discrete time analogue to an expected integral of $Y$'s maximum process plugged into $f$. These integral representation together with the monotonicity properties of $f$ postulated in Assumption \ref{bedingung} first ensured that $\tau_{ \bar x}$ was the optimizer amongst all threshold times and later on also helped to ensure maximality in the initial problem. Now our course of action is as follows: We first simply assume that there is a function $f$ as in Assumption \ref{bedingung} that enables an integral type maximum representation similar to the one above, then we are able to follow the same line of argument as in the discrete time case. After that we first discuss existence of such a function and also in the case of L\'evy processes give a construction (analogous to the discrete time case with help of the ladder height process). We follow the line of argument used in \cite{ChristensenSohr} for a related problem arising in impulse control.

\subsection{Notation and Prerequisites}\label{notationen}

Let $X$ be a strong Markov process on $\R$ with a  right continuous filtration $\mathcal F:=\left(\mathcal F_t\right)_{t\geq 0}$ such that $\mathcal F_0$ is complete. %Denote the underlying probability space with $\left(\Omega, \mathcal P, \Sigma\right)$ and for all $x \in \R$ define 
Further we work with the usual associated family of measures $\Pro_x\left(\cdot\right):=\Pro\left(\cdot|X_0=x\right)$ and for each $x \in \R$ let $\E_x$ be the expectation operator associated to $\Pro_x$. %and shortly write $\Pro:=\Pro_0$ and $\E:=\E_0$.\\
Let $\gamma$ and $h$ be real functions and assume:
\begin{assumption}\label{bedingungstetig}
	\begin{enumerate}
		\item $\gamma$ is non-decreasing and differentiable. 
		\item\label{hassumption} $h$ is non-negative, continuous and for all $x,y \in \R$ with $x<y$, we have $$\E_x\left(\int_0^{\tau_y}h\left(X_s\right)\ ds\right)< \infty.$$ 
	\end{enumerate}
\end{assumption}
Define the set $\mathcal T$ as the set of all stopping times $\tau$ with $\E_x\left(\tau \right)<\infty$ and $ \E_x\left(\int_0^\tau h\left(X_s\right)~ds\right)~<~\infty$ for all $x \in \R$.
We look at the stopping problem
$$
\mathcal V(x):=\sup_{\tau \in \mathcal T}\E_x\left(\gamma\left(X_\tau\right)-\int_{0}^{\tau}h\left(X_t\right)  \ dt \right) 
$$ for all $x \in \R$. With similar steps as in the discrete time case, we develop sufficient conditions under that a threshold time, whose threshold is given as the root of a function $f$, is optimal.
Since in our setting there occurs no killing and/or discounting, an issue is the lack of resolvents that are often used to obtain maximum representations, see e.g. \cite{mordeckisalminen} \cite{novikovshiryaevstopping07}, \cite{surya07} and \cite{CST13}. Instead, we first fix $\bar y \in \R$ and assume a maximum representation analogous to the one in the discrete time case, and justify that it suffices to only work with stopping times bounded by first entry times to intervals of the form $\left(-\infty, \bar y\right]$ if $\bar y $ is large enough. 
\\
\begin{assumption}\label{maxdst} We assume that there is a function $f$ such that: \begin{enumerate}
		
		\item For all $x  \leq {\bar y}$ $$
		\gamma\left(x\right)= - \E_{x}\left[\int_{0}^{\tau_{\bar y}}f\left(\sup_{ r\leq t} X_r\right)dt\right]+\E_{x}\left[\gamma\left(X_{\tau_{{\bar y}}}\right)-\int_{{ 0}}^{\tau_{\bar y}} h\left(X_s\right) ds\right].$$
		\item The function $f$ has a unique root $\bar x \in \R$ and is  strictly decreasing on $[\bar x,\infty)$. 
	\end{enumerate}
\end{assumption}

\begin{remark}
	
	Whenever the running maximum occurs, we tacitly assume $\Pro_x$ to be $\Pro_{\left(x,x\right)}$, the measure corresponding to the two dimensional Markov process $\left(X_t,\sup_{ r\leq t }X_r\right)_{t\geq 0}$ started in $\left(x,x\right)$. So we are still able to exploit the Markovian structure. 
\end{remark}

%To show optimality of a threshold time, we follow the approach used in $\cite{ChristensenSohr}$ in a slightly different setting.  

\begin{defn}
	Call a stopping time $\tau$ upper regular if there is a value $\bar y \in \R$ such that $\tau$ is under all $\Pro_y$ a.s. bounded by the first entry time of $X$ in $[\bar y,\infty)$. Define $\mathcal U:=\{\tau \in \mathcal T|\tau \text{ is upper regular}\}$. 
\end{defn}
\begin{lemma}\label{regularlemma}
	For all $x,y\in \R$ with $x \geq  y$ we have $$\mathcal V (x)=\sup_{\tau \in \mathcal U}\E_x\left(\gamma\left(X_\tau\right)-\int_0^\tau h\left(X_s\right)  \ ds\right)$$
\end{lemma}

\begin{proof}
	Take a $\tau \in \mathcal T$ that is $\epsilon$-optimal. Set for all $n \in \N$ \[\sigma_n:=\tau \wedge \inf\{t\geq 0 |X_t \geq n+X_0\}   \] 
	We have $\sigma_n\rightarrow \tau$ a.s. under all $\Pro_z$ and, since $ \int_0^\tau h\left(X_s\right)  \ ds$ works as an integrable majorant, we get with dominated convergence \[
	\E_x\left(\int_0^{\tau} h\left(X_s\right) \ ds\right)=\lim_{n\rightarrow \infty}\E_x\left(\int_0^{\sigma_n}h\left(X_s\right)  \  ds\right)
	\]
	lastly, because both $\gamma$ and $\left(\sigma_n\right)_{n\in \N}$ are (a.s. under all $\Pro_z$) monotone, monotone convergence yields \begin{align*}
	\E_x\left(\gamma\left(X_\tau\right)\right)&=\E_x\left(\lim_{n\rightarrow \infty}\gamma\left(X_\tau\right)\wedge \gamma\left(X_{\sigma_n}\right)\right)\\&= \lim_{n\rightarrow \infty}\E_x\left(\gamma\left(X_\tau\right)\wedge \gamma\left(X_{\sigma_n}\right)\right)\\&\leq\lim_{n\rightarrow \infty}\E_x\left( \gamma\left(X_{\sigma_n}\right)\right)
	\end{align*}
	Altogether this yields
	\[
	\E_x\left(\gamma\left(X_\tau\right)-\int_0^{\tau}h\left(X_s\right)  \ ds\right)\leq\lim_{n\rightarrow \infty}\E_x\left(\gamma\left(X_{\sigma_n}\right)-\int_0^{\sigma_n} h\left(X_{s} \right)\ ds\right).
	\]
\end{proof}

\begin{lemma}\label{39}
	Let $x,\bar y \in \R$ with $x\leq \bar y$. For $a_{\bar y}:=\bar y \wedge \bar x$ holds
	$$\sup_{  \tau \leq  \tau_{\bar y}}\E_{x}\left[\int_{0}^{\tau}f\left(\sup_{ r\leq t} X_r\right)dt\right]=\E_{x}\left[\int_{0}^{\tau_{a_{\bar y}}}f\left(\sup_{r \leq t} X_r\right)dt\right]$$ 
\end{lemma}
\begin{proof}
	This is a direct consequence of the properties of $f$ posed upon it in Assumption \ref{maxdst} and the monotonicity of $\sup_{r \leq t}X_r$. 
\end{proof}

\begin{thm}\label{stoppinghauptsatz}
	For all $x \in \R$ holds $$ \mathcal V(x)=\E_x\left(\gamma\left(X_{\tau_{\bar x}}\right)-\int_{0}^{\tau_{\bar x}}\left(h\left(X_s\right) \right)\ ds\right). $$
	
\end{thm}
\begin{proof}
	We define for all $x\in \R$ $$\tilde g\left(x\right):=\E_x\left(\gamma\left(X_{\tau_{\bar x}}\right)-\int_{0}^{\tau_{\bar x}}h\left(X_s\right) \ ds\right).$$ One immediately sees that $ \mathcal V \geq \tilde g \geq \gamma$.
	Let $x \in \R$. Lemma \ref{regularlemma} tells us that it suffices to show $$\tilde g\left(x\right)\geq \sup_{  \tau \in \mathcal U} \E_x\left(\gamma\left(X_\tau\right)-\int_0^{\tau} h\left(X_s\right)   \ ds\right)$$ in order to prove $\mathcal V=\tilde g$.
	\\
	Let $\tau\in \mathcal U$ be an upper regular stopping time and fix an ${\bar y}>\bar x, x$ such that $\tau\leq \tau_{\bar y}$ $\Pro_x$ a.s. Then we have
	
	%where we use the monotonicity assumption posed upon $f$. Now we are ready to tackle the $h$-superharmonicity. Using the previous calculation we get: 
	\begin{align*}
	\E_x &\left[\gamma\left(X_\tau\right)-\int_0^{\tau}h\left(X_s\right) \ d s\right]
	\\\stackrel{\ref{maxdst}}{=}&\E_x\left\{-\E_{X_\tau}\left[\int_{0}^{\tau_{\bar y}}f\left(\sup_{ r\leq t}X_r\right)dt\right]\right.
	\\&\left.+\E_{X_\tau}\left[\gamma\left(X_{\tau_{{\bar y}}}\right)-\int_{{ 0}}^{\tau_{\bar y}} h\left(X_s\right) \ ds\right] -\left[\int_0^{\tau}h\left(X_s\right)\ d s\right]\right\}
	\\=&-\E_x\left\{\E_{X_\tau}\left[\int_{0}^{\tau_{\bar y}}f\left(\sup_{ r\leq t}X_r\right)dt\right]\right\}+\E_{x}\left[\gamma\left(X_{\tau_{{\bar y}}}\right)-\int_{0}^{\tau_{\bar y}} h\left(X_s\right)\ ds \right]
	%\\=&\E_x\left\{{1}_{\{\tau \leq \tau_{\underline x}\}}\E_{X_\tau}\left[\int_{0}^{\tau_{\bar y}}f\left(\sup_{ r\leq t}X_r\right)dt\right]\right\}
	%\\&+\E_x\left\{{1}_{\{\tau > \tau_{\underline x}\}}\E_{X_\tau}\left[\int_{0}^{\tau_{\bar y}}f\left(\sup_{ r\leq t}X_r\right)dt\right]\right\}+\E_{x}\left[\gamma\left(X_{\tau_{{\bar y}}}\right)-\int_{0}^{\tau_{\bar y}} h\left(X_s\right)\ ds \right]
	\\=&-\E_x\left\{\E_{x}\left[\int_{\tau}^{\tau_{\bar y}}f\left(\sup_{ r\leq t}X_r\right)dt|\mathcal F_\tau \right]\right\}+\E_{x}\left[\gamma\left(X_{\tau_{{\bar y}}}\right)-\int_{0}^{\tau_{\bar y}} h\left(X_s\right)\ ds \right]
	%\\=&\E_x\left\{{1}_{\{\tau \leq \tau_{\underline x}\}}\E_{x}\left[\int_{\tau}^{\tau_{\bar y}}f\left(\sup_{ r\leq t}X_r\right)dt|\mathcal F_\tau \right]\right\}
	%\\&+\E_x\left\{{1}_{\{\tau > \tau_{\underline x}\}}\E_{x}\left[\int_{\tau}^{\tau_{\bar y}}f\left(\sup_{ r\leq t}X_r\right)dt|\mathcal F_\tau \right]\right\}+\E_{x}\left[\gamma\left(X_{\tau_{{\bar y}}}\right)-\int_{0}^{\tau_{\bar y}} h\left(X_s\right)\ ds \right]
	%\\=&\E_x\left\{\E_{x}\left[{1}_{\{\tau \leq \tau_{\underline x}\}}\int_{\tau}^{\tau_{\bar y}}f\left(\sup_{ r\leq t}X_r\right)dt|\mathcal F_\tau \right]\right\}
	%\\&+\E_x\left\{\E_{x}\left[{1}_{\{\tau > \tau_{\underline x}\}}\int_{\tau}^{\tau_{\bar y}}f\left(\sup_{ r\leq t}X_r\right)dt \right]|\mathcal F_\tau \right\}
	%\\&+\E_{x}\left[\gamma\left(X_{\tau_{{\bar y}}}\right)-\int_{0}^{\tau_{\bar y}} h\left(X_s\right) \ ds \right]
	\\\stackrel{\ref{maxdst}}{\leq}&-\E_{x}\left[\int_{\tau_{\bar x}}^{\tau_{\bar y}}f\left(\sup_{ r\leq t}X_r\right)dt\right]
	+\E_{x}\left[\gamma\left(X_{\tau_{{\bar y}}}\right)-\int_{0}^{\tau_{\bar y}} h\left(X_s\right)\ ds \right]
	\\\stackrel{\ref{maxdst}}{=}&-\E_{x}\left[\int_{\tau_{\bar x}}^{\tau_{\bar y}}f\left(\sup_{ r\leq t}X_r\right)dt\right] +\gamma\left(x\right)+\E_x\left[\int_{{ 0}}^{\tau_{\bar y}}f\left(\sup_{r\leq t}X_{r}\right) \ dt\right]
	\\=&\gamma\left(x\right)+\E_{x}\left[\int^{\tau_{\bar x}}_0  f\left(\sup_{ r\leq t}X_r\right)dt \right]
	\\\stackrel{\ref{maxdst}}{=} & \E_{x}\left[\gamma\left(X_{\tau_{ \bar x}}\right)-\int_0^{\tau_{\bar x}} h\left(X_s\right)  \ ds\right].
	\end{align*}
	
\end{proof}

\subsection{Discussion of Assumption \ref{maxdst}}

The key part of our solution is Assumption \ref{maxdst}, namely the existence of a function $f$ such that for all $x, \bar y \in \R$ with $x~\leq~{\bar y}$ we have \begin{align}
\gamma\left(x\right)=-\E_{x}\left[\int_{0}^{\tau_{\bar y}}f\left(\sup_{ r\leq t} X_r\right)dt\right]+\E_{x}\left[\gamma\left(X_{\tau_{{\bar y}}}\right)-\int_{{ 0}}^{\tau_{\bar y}} h\left(X_s\right) ds\right].\label{maxdst2}
\end{align}
If  $f$ has a unique root and is monotonous after, the stopping problem we solved in the previous section has a threshold time as an optimizer. This section aims to discuss existence of such a maximum representation.\\
First we give a brief overview on literature regarding maximum representations for general Markov processes, thereafter we focus on explicit obtainability in the case that the underlying process is a L\'evy process. This construction has been developed in \cite{ChristensenSohr} in a similar setting, therefore we will leave out the proofs if they are already given in said work. 
\subsubsection{General Remarks}
The comprehensive intuition for the use of maximum representations in optimal stopping arises from the fact that under quite general assumptions superharmonic functions can be characterized as expectations of non-negative functions of the maximum process, see \cite{follmerknispel} and the references therein for a proof from a potential theoretic perspective. That the value functions of a stopping problem is basically the smallest superharmonic majorant of the pay-off function, suggests the assumption that, when any maximum representation of the pay-off function in terms of a function $f$ can be found, a candidate for the smallest superharmonic majorant is the maximum representation in terms of $f^+$, the positive part of aforementioned  $f$. And indeed in the discounted case in \cite{CST13}, it is shown that, if the pay-off function has a kind of maximum representation and a suitable representing function $f$ is shaped nicely, indeed $f^+$ is  the representing function of the value function's maximum representation. In \cite{MR3620898} a suggestion can be found to get such a representation: If somehow, there is a terminal state $\zeta$ of $X$, such that neither $X_\zeta$ nor the running maximum $M_\zeta$ are degenerate, then one can first find a representation of $\gamma$ in terms of $X_\zeta$ (for example with Dynkin's formula or using resolvents) and later condition on $M_\zeta$. This can be used to show existence of a maximum representation. Nevertheless this approach is not applicable here and suffers from the lack of explicity. 
\subsubsection{L\'evy Processes}
Therefore next, we take some steps towards an explicit obtainability of a maximum representation under the assumption that $X$ is a L\'evy process with $\E(X_1)>0$. 
First, we will give sufficient conditions for an $f$ as in \eqref{maxdst2} to exist and thereafter take some steps to the (semi-)explicit obtainability in interesting cases. This theory was developed in \cite{ChristensenSohr} and the detailed proofs can be found therein. \\
%The approach heavyly utilizes the ascending as well as the descending ladder height process of $x$, the definition and some useful properties can be found in the appendix \ref{ladderheightprocess}. 
We fix a $\bar y \in \R$ through the section: Further let $H$ denote the ascending ladder height process of $X$ and $H^\downarrow$ denote the descending ladder height process of $X$. As these corresponding ladder time processes are only defined up to a multiplicative constant, we choose them such that $\E(\tau_x)=\E(\inf\{t\geq 0|H_t\geq x\})$ and $H^\downarrow$ accordingly. \\

\begin{lemma}\label{dachfunktion}
	For each positive function $g$ define for all $y \geq 0$
	\begin{align*}
	\hat g\left(y\right) :=\E_y\left(\int\limits_0^\infty g\left( H^{\downarrow}_t\right) \ dt \right).
	\end{align*}
	Then for all $x\leq y$
	\begin{align*}
	\E_x \left(\int_0^{\tau_y} g\left(X_t\right) \ dt \right)=\E_x\left( \int_0^{\hat \tau_y}\hat g \left(H_t\right) \ dt\right).
	\end{align*}
\end{lemma}
%\begin{proof}
%	This  result is a reformulation of Exercise 7.10 in \cite{kypri} and originates in \cite{silverstein}. 
%\end{proof}
%\begin{remark} 
%	The process $H^\downarrow$ acts in law like a killed subordinator, see Theorem 6.9 in \cite{kypri}.
%\end{remark}

In the following the operator $A_H$ is defined as the generalized extended generator of $H$ and we assume $\gamma$ is in the range of $A_H$. This enables us to set
\begin{align}
f:=\left(A_H\gamma +\hat h\right)\label{fdef}
\end{align}
and an application of Dynkin's formula yields
\begin{align*} 
\E_x\left[\gamma\left(X_{\tau_y}\right)-\int_0^{\tau_y} h\left(X_s\right) \ ds\right]=\E_x\left[\int_0^{\hat \tau_y}f\left(H_s\right)ds\right]+\gamma\left(x\right)\label{777}
\end{align*} 	for all $x,y<\ \bar y\in \R$.
%\begin{proof}

%	For all $x,y< \ \bar y\in \R$ we have using Dynkin's formula and Lemma \ref{dachfunktion}:
%	\begin{align*}
%		\E_x[&\gamma\left(X_{\tau_y}\right)-\int_0^{\tau_y}\left(h\left(X_s\right) \right) \ ds]\\&=\E_x[\gamma\left(H_{\hat \tau_y}\right)-\int_0^{\hat \tau_y}\left(\hat h\left(H_s\right) \right) \ ds]
%		\\&=\E_x[\int_0^{\hat \tau_y}\left(A_H\gamma+\hat h\right)\left(H_s\right)ds]+\gamma\left(x\right)
%		\\&\stackrel{\ref{fdef}}{=}\E_x[\int_0^{\hat \tau_y}f\left(H_s\right)ds]+\gamma\left(x\right)
%	\end{align*}
%\end{proof}
Now by algebraic induction we get
for all $x < y$ 
\begin{align*}
\E_x\left[\int_{0}^{\hat \tau_{y}}f\left(H_s\right) \ ds\right]{=}\E_x\left[\int_{0}^{\tau_{ y}}f\left(\sup_{r\leq t} X_r\right)\ ds \right],	\end{align*}
where $\hta_x$ denotes the first entry time of $H$ in $[x,\infty)$.
%\begin{proof}
%	
%	This can be proven by algebraic induction: Wald's identity shows $$\E_x\left(\hta_y\right)=\E\left(L^{-1}_1\right)\E_x\left(\tau_y\right)=\E_x\left(\tau_y\right),$$ hence the claim holds for indicator functions of the form $ 1_{[x,y]}$ and with the Markov property this extends to indicator functions of general intervals. This carries over to simple positive functions due to linearity and with Fatou's lemma to general positive functions. Decomposition in a positive and a negative part yields the claim for general regular functions.
%\end{proof}

%	For all $x<{\bar y}$ holds $$\gamma\left(x\right)=\E_x\left[-\int_{0}^{\tau_{\bar y}}\frac{1}{c}\left(A_H\gamma+\hat h\right)\left(\sup_{r\leq t}X_r\right) d t\right]+\E_x\left[\gamma\left(X_{\tau_{\bar y}}\right)-\int_{0}^{\tau_{\bar y}} h\left(X_s\right)ds\right]$$
%\end{lemma}
%\begin{proof}
%	We fix a  $x<{\bar y}$ and get\begin{align*}
%	\gamma\left(x\right)&=\E_x\left[\gamma\left(H_{\hat \tau_{\bar y}}\right) -\int_0^{\hta_{\bar y}}A_H\gamma\left(H_s\right) d s\right]\\&=\E_x\left[-\int_{0}^{\hat \tau_{\bar y}}\left(A_H\gamma+\hat h\right)\left(H_s\right) d s\right]+\E_x\left[\gamma\left(X_{\tau_{\bar y}}\right)-\int_{0}^{\tau_{\bar y}} h\left(X_s\right)ds\right]\\&\stackrel{\ref{supdarstellung}}{=}\E_x\left[-\int_{0}^{\tau_{\bar y}}\frac{1}{c}\left(A_H\gamma+\hat h\right)\left(\sup_{r\leq t}X_r\right) d t\right]+\E_x\left[\gamma\left(X_{\tau_{\bar y}}\right)-\int_{0}^{\tau_{\bar y}} h\left(X_s\right)ds\right]\end{align*} 

%\end{proof}
Putting all this together yields the desired result:
\begin{lemma}\label{maxdarstellungshauptlemma} For all $x  \leq {\bar y}$ holds $$
	\gamma\left(x\right)=-\E_{x}\left[\int_{0}^{\tau_{\bar y}}f\left(\sup_{ r\leq t} X_r\right)dt\right]+\E_{x}\left[\gamma\left(X_{\tau_{{\bar y}}}\right)-\int_{{ 0}}^{\tau_{\bar y}} h\left(X_s\right) ds\right].$$
\end{lemma}
%\begin{proof}
%	Lemma \ref{777} with $y= \bar y$ combined with Lemma \ref{dachfunktion} yields
%	\begin{align*}
%		\gamma\left(x\right)
%		\stackrel{\ref{777}}{=}&-\E_{x}\left[\int_{0}^{\hat \tau_{\bar y}}f\left(H_s\right)ds\right]  +  \E_{x}\left[\gamma\left(X_{\tau_{{\bar y}}}\right)-\int_{{ 0}}^{\hat\tau_{\bar y}} \hat h\left(H_s\right) ds\right]
%	\\ \stackrel{\ref{dachfunktion}}{=}&-\E_{x}\left[\int_{0}^{\tau_{\bar y}}f\left(\sup_{ r\leq t} X_r\right)dt\right]+\E_{x}\left[\gamma\left(X_{\tau_{{\bar y}}}\right)-\int_{{ 0}}^{\tau_{\bar y}} h\left(X_s\right) ds\right]
%\end{align*}
%\end{proof} 

\section{CONNECTION OF THE PROBLEMS}

This similarity of both problems' structures, especially the similarity in the functions $f$ that determine the optimal threshold, suggests that the solution of the continuous problem can be found via discretization. Hence, this section aims to give conditions under that the solution of embedded discrete problems converges to the solution of the continuous problem. \\
Again let $X$ be a strong Markov process on $\R$ with a  right continuous filtration $\mathcal F:=\left(\mathcal F_t\right)_{t\geq 0}$ such that $\mathcal F_0$ is complete. As before we look at the continuous time stopping problem 
$$
\mathcal V(x)=\sup_{\tau \in \mathcal T}\E_x\left(\gamma\left(X_\tau\right)-\int_{0}^{\tau}h\left(X_t\right)  \ dt \right) 
$$ for all $x \in \R$. Further we still work with Assumption \ref{bedingungstetig}. 
\begin{defn}\label{discretizationdef}
	We call a sequence of ascending sequences of stopping times $\left((\tau^n_k)_{k \in \N}\right)_{n\in \N}$ a suitable discretization if \begin{itemize}
		\item $\{\tau^n_k|k \in \N\}\subseteq\{\tau^{n+1}_k|k \in \N\}$ for all $n\in \N$. 
		\item For each $n\in \N$ the process $Y^n$ defined by $Y^n_k:=X_{\tau_k^n}$ for all $k\in \N$ is a Markov process. 
		\item For each $X$-stopping time $\tau$, for each $n\in \N$ there is a $Y^n$-stopping time $\ceil \tau^n  $ such that for all $n \in \N$ we have  $\tau \leq \ceil \tau^{n+1} \leq \ceil \tau^{ n}$ and  a.s.$$\lim_{n\rightarrow \infty} \ceil \tau^n = \tau.$$
	\end{itemize}
\end{defn}
We say a suitable discretization  $\left((\tau^n_k)_{k \in \N}\right)_{n\in \N}$ harmonizes with $\gamma$, if for all $x \in \R$ and all $X$-stopping times $\tau$
$$ \lim_{n\rightarrow \infty} \E_x\left(\gamma(X_{\ceil \tau^ n})\right)=\E_x\left(\gamma(X_{\tau})\right).
$$
We say a suitable discretization  $\left((\tau^n_k)_{k \in \N}\right)_{n\in \N}$ harmonizes with $h$, if for each $n\in \N$ there is a function $\ceil h^n$ such that for all $x \in \R$ and all $X$-stopping times $\tau$
$$ \lim_{n\rightarrow \infty} \E_x\left(\sum_{i=1}^{\ceil \tau^n}\ceil h^n(Y^n_i)\right)=\E_x\left(\int_0^\tau h(X_s) ds \right).
$$
With these definitions at hand we immediately get the following result: 
\begin{thm}
	Assume $\left((\tau^n_k)_{k \in \N}\right)_{n\in \N}$ is a suitable discretization. 
	Using the definitions in \ref{discretizationdef}, define for all $n \in \N$ 
	$$\mathcal V_n(x):= \sup_{\tau} \E_x\left( \gamma(Y_\tau)-\sum_{i=1}^\tau \ceil{h}_n(Y_i) \right), $$ where the supremum is taken over all $Y_n$-stopping times. Assume each $\mathcal V_n$ has an optimal stopping time and let $\mathcal S_n:=\{x\in \R|\mathcal V_n(x)=\gamma(x)\}$ be the stopping set of $V_n$ and $\mathcal S:=\{x\in \R|\mathcal V(x)=\gamma(x)\}$ the one for $\mathcal V$. 
	Then point-wise $$ \mathcal V_n \nearrow \mathcal V
	$$ and $$\bigcap_{n\in \N} \mathcal S_n = \mathcal S.
	$$
\end{thm}
\begin{proof}
	Observe that $$\gamma \leq\mathcal V_n \leq \mathcal V_{n+1} \leq  \mathcal V $$ for all $n \in \N_0$. Hence $\lim\limits_{n \rightarrow \infty} \mathcal V_n$ exists and $\lim\limits_{n \rightarrow \infty}\mathcal  V_n \leq \mathcal V$. Let $x \in \R$. We only treat the case that $\mathcal V(x)<\infty$, the case $\mathcal V(x)=\infty$ works analogously with the obvious alterations. 
	For each $ x \in \R$ and each $ \epsilon \geq 0$ there is an $X$-stopping time $\tau $ such that $$
	\mathcal V\left(x\right)\leq \E_x\left(\gamma \left(X_\tau\right)-\int\limits_0^{\tau}h\left(Y_s\right)ds\right)+ \frac{\epsilon}{2}
	$$ and an $N \in \N$ such that $$\E_x\left(\gamma \left(X_\tau \right)-\int\limits_0^{\tau}h\left(X_s\right)ds \right) \leq \E_x\left(\gamma \left(X_{\tau^N}\right)-\sum\limits_{i=1}^{\ceil \tau^N}\ceil h^n\left(Y_i\right)\right)+ \frac{\epsilon}{2} $$
	Hence we get \begin{align*}
	\mathcal V\left(x\right)&\leq \E_x\left(\gamma \left(X_\tau\right)-\int\limits_0^{\tau}h\left(X_s\right)ds \right)+ \frac{\epsilon}{2} \\&\leq \E_x\left(\gamma \left(X_{\tau^N}\right)-\sum\limits_{i=1}^{\ceil \tau^N}\ceil h^n\left(Y_i\right) \right)+ \epsilon \\& \leq \mathcal V_N\left(x\right) +\epsilon\\& \leq \lim\limits_{n\rightarrow \infty } \mathcal V_n\left(x\right)+\epsilon .
	\end{align*}	
	Assume $x\in \mathcal S_n$ for all $n\in \N$. Then $\gamma(x)=\lim_{n\rightarrow \infty}\mathcal V_n(x))=\mathcal V(x)$.
\end{proof}

\begin{remark}
	Note that these proofs don't even depend on one dimensionality. However, in the one-dimensional one-sided case these results imply that if the discretized problems are one-sided, the continuous one also is and additionally the threshold of the continuous problem is the monotone limit of the discretized problems' thresholds. 
\end{remark}

Now that we have given conditions under that suitable discrete problems approximate the continuous one in the right way, depending on the process and the functions $h$ and $\gamma$, we have to find suitable discretizations. 

\subsection{Non-random Discretization for L\'evy Processes}
For the sake of simplicity we assume $X$ to be a L\'evy process with $\E_0(|X_1|)<\infty$, nevertheless we want to mention that under the right assumptions the results remain to hold for more general processes.  
One of the simplest imaginable discretizations is  $\left((\tau^n_k)_{k \in \N}\right)_{n\in \N}=\left((\frac{k}{2	n})_{k \in \N}\right)_{n\in \N}.$
We assume $\gamma$ to be differentiable and that there is an $M \in \R$ such that $|\gamma| , |h|<M$. For all $x \in \R$ and all $n \in \N$ we set $$\ceil{h}_n(x):=\E_x\left(\int_{{ 0}}^{\frac{1}{2^n}}h(X_s) ds\right)$$ and for all $X$-stopping times $\tau$ and all $n \in \N$ $$
\ceil \tau^n:= \frac{\lfloor 1+\tau2^{n} \rfloor}{2^n} .
$$
Now $\E_x|X_1|<\infty$ implies $\E_x\left(\sup\limits_{t \leq 1}|X_t|\right) < \infty$, see  \cite{gut}, and hence 
\begin{align*}
\left|\E_x\left( \gamma \left(X_{\tau^n}\right)-\gamma \left(X_{\tau}\right)\right)\right|\leq M\E_x\left|X_\tau-X_{\ceil \tau^n} \right|\leq M \E_x(\sup\limits_{ t \leq \frac{1}{2^n}}|X_t|) \stackrel{n\rightarrow \infty }{\rightarrow} 0 
\end{align*} as well as
\begin{align*}
\left|\E_x\left( \sum\limits_{i=1}^{\ceil \tau^n}\ceil h^n\left(Y_i\right)-\int\limits_0^{ \tau}h\left(Y_s\right)ds\right)\right| \leq M\frac{1}{2^n} \stackrel{n\rightarrow \infty }{\rightarrow} 0. 
\end{align*}
This yields that the non-random discretization is a suitable discretization that harmonizes with $\gamma$ and $h$. We want to remark two things:
First, one can weaken the boundedness assumptions to $\gamma$ and $h$, as shown in \cite{MR1626790} by approximation with bounded functions. 
Second, the trade-off for the relatively strong restrictions on the functions is a pretty high compatibility of discrete and continuous problems. When one assumes concavity of $\gamma$ and monotonicity of $h$, the functions $\ceil{h}^n$ are also monotone, the $Y^n$ are random walks and hence this immediately yields that the $\mathcal S_n$ are one sided intervals as discussed in \ref{random walk}.

\subsection{Spatial Discretization}
Another approach which nicely stresses out the connection of the representing functions $f$ of the discrete and continuous problem is to use a separation of the state space instead of the time axis. 
For each $k,n\in \N$,  define $\tau^n_0:=0$ and $\tau^n_k:=\inf\left\{t\geq \tau^n_{k-1}\Big| X_t \in \R \setminus \left[\frac{\lfloor 2^{n}X_{\tau^n_{k-1}} \rfloor}{2^n},\frac{\lfloor 1+2^{n}X_{\tau^n_{k-1}} \rfloor}{2^n}\right)\right\}$
Again using the obvious choice $$\ceil{h}^n(x):=\E_x\left(\int_{{ 0}}^{\tau_1^n}h(X_s) ds\right)$$ we see that  $\left((\tau^n_k)_{k \in \N}\right)_{n\in \N}$. is a suitable discretization and if we assume $h$ to be continuous, $\gamma$ to be smooth enough to be in the range of the extended generator of $X$ and that for large enough $n$ Dynkin's formula is applicable to $\tau_1^n$, we get 
\begin{align*}
\left|\E_x\left( \gamma \left(X_{\tau^n}\right)-\gamma \left(X_{\tau}\right)\right)\right|=
\E_x\left( \int\limits_{\tau}^{\ceil \tau^n} A \gamma(X_s) ds  \right)\leq  \E_x\left(\E_{X_\tau}\left( \int\limits_{0}^{ \tau^n_1} A \gamma(X_s) ds \right) \right) 
\end{align*} and also
\begin{align*}
\left|\E_x\left( \sum\limits_{i=1}^{\ceil \tau^n}\ceil h^n\left(Y_i\right)-\int\limits_0^{ \tau}h\left(Y_s\right)ds\right)\right| \leq  \E_x\left(\E_{X_\tau}\left( \int\limits_{0}^{ \tau^n_1} h(X_s) ds \right) \right) .
\end{align*}

Now only some boundedness assumptions are needed to ensure that both this terms converge to zero. For example known properties of the stopping region, like boundedness of the continuation region and monotonicity of $\gamma$ and $h$ or if one can justify only to consider bounded stopping times may help here. Instead of going into more detail regarding this, we want to emphasize that with this discretization one can nicely see how the representing functions of the discrete and the continuous time problems are connected. If we denote the representing function of the discrete time problem of stopping $\gamma(Y^n_k)-\sum_{i=1}^{\tau^n_k}\ceil{h}^n(Y^n_i)$ as defined in \eqref{bedingunggleichung} in Section 2 with $f_n$, we see that, if we first assume $h=0$ and for the sake of notational simplicity also $x\in \frac{1}{2^n}\N$, we have \begin{align*}
f_n(x)&=\frac{ \E_x\left(\gamma(Y^n_{\tau^+})\right)-\gamma(x)}{\E(\tau^n_1)}\\
&=\frac{ \E_x\left(\gamma(H_{\hat\tau_{x+\frac{1}{2^n}}})\right)-\gamma(x)}{\E(\tau^n_1)}
\\&\rightarrow A_H \gamma (x),
\end{align*} 
provided $\gamma$ is in the range of the extended generator of $H$. 
To treat the case of arbitrary $h$ we use Lemma \ref{dachfunktion}. With the notations used therein we see that, again for $x\in \frac{1}{2^n}\N$,  \begin{align*}
\frac{\E_x(\ceil h(Y_{\tau^+}^n))}{\E_x(\tau_1^n)}
&=\frac{\E_x(\int_0^{\tau_{x+\frac{1}{2^n}}}h(X_s) d s)}{\E_x(\tau_{x+\frac{1}{2^n}})}
\\&=\frac{\E_x(\int_0^{\hat\tau_{x+\frac{1}{2^n}}}\hat h(H_s) d s)}{\E_x(\hat\tau_{x+\frac{1}{2^n}})}
\\& \rightarrow \hat h(x),
\end{align*}
provided $h$ is smooth enough. 

%\appendix
%\section{Ladder Height Process}\label{ladderheightprocess}
%This is en excerpt of \cite{ChristensenSohr}.
%\begin{defn}[Timeshift]
%We assume $\Omega$ to be the c\`adl\'ag functions equipped with the Borel sigma field and set for all  $$\theta_t\left(\omega\right):=[0,\infty\right)\rightarrow \R; \ s \mapsto \omega\left(t+s\right)$$ 
%\end{defn}
%\begin{lemma}[and Definition]
% For each increasing sequence of stopping times $\left(\tau_i\right)_{i \in \N}$ there is a sequence of stopping times $\left(\Delta \tau_i\right)_{i \in \N}$ such that $$\tau_i=\tau_{i-1}+\Delta \tau_i \circ \theta_{\tau_{i-1}} \\ \forall i \in \N$$. 
%\end{lemma}
%\begin{defn}\label{ladderheightdefinition}
%	Let $L$ be a local time at the maximum as defined in Definition 6.1 in \cite{kypri} and $H$ defined by $$H_t:=X_{L^{-1}\left(t\right)}$$ for all $t \geq 0$ the ladder height process. $\left(H,L^{-1}\right)$ is a L\'evy process, even a bivariate subordinator. With Wald's equation we have $\E\left(L^{-1}\left(\tau_x\right)\right)= \E\left(L_1^{-1}\right)^{-1}\E\left(\tau_x\right).$ Since $L$ is only defined up to a multiplicative constant, w.o.l.g. we choose $L$ such that $\E\left(L_1^{-1}\right)=1$ and hence $\E\left(L^{-1}\left(\tau_x\right)\right)=\E\left(\tau_x\right)$ for all $x \in \R$.   
%	Further, we set $$\hat{\tau}_x:=L^{-1}\left(\tau_x\right)=\inf\{t\geq 0|H_t\geq x\}$$ for all $x \in \R$. \\
%In the same way we define the descending ladder height process $H^\downarrow$ as the ladder height process of $-X$.\\
%Lastly, we define $A_H$ as the extended generator of $H$.  
%\end{defn}

\newpage

\small \selectfont

\bibliography{Christensen_Sohr_General_Stopping}

\end{document}